\documentclass[a4paper,12pt,final]{amsart}
\usepackage{times,a4wide,mathrsfs}
\usepackage{amsmath,amssymb}

\newcommand{\D}{\mathbb{D}}
\newcommand{\RR}{\mathbb{R}}
\newcommand{\C}{\mathbb{C}}
\newcommand{\T}{\mathbb{T}}

\newcommand{\Cl}{{\mathscr{C}}}
\newcommand{\Hi}{{\mathscr{H}}^\infty}
\newcommand{\Ht}{{\mathscr{H}}^2}
\newcommand{\Hp}{{\mathscr{H}}^p}
\newcommand{\Ha}{{\mathscr{H}}^{2^{\alpha+1}}}
\newcommand{\Hab}{{\mathscr{H}}^{2^{[\alpha]+2}(2+[\alpha]-\alpha)^{-1}}}
\newcommand{\Da}{{\mathscr{D}}_\alpha}
\newcommand{\Din}{{\mathscr{D}}_\infty}
\newcommand{\Dx}{{\mathscr{D}}_x}

\newtheorem{theorem}{Theorem}
\newtheorem*{theorem1p}{Theorem 1'}
\newtheorem*{theorem2p}{Theorem 2'}
\newtheorem{lemma}{Lemma}

\begin{document}

\title[Zeros of functions in Hilbert spaces of Dirichlet series]{Zeros of
functions in Hilbert spaces of Dirichlet series}

%\subjclass[2000]{32A05, 43A46}

%\author[E. Saksman]{Eero Saksman}
%\address{Department of Mathematics and Statistics, University of Helsinki,
%P.~O.~Box 68 (Gustaf H\"{a}llstr\"{o}min Katu 2B), 00014 University
%of Helsinki, Finland} \email{eero.saksman@helsinki.fi}

%\thanks{The first author was supported by the Academy of Finland, projects
%no.\ 113826 and 118765. }

%    Information for second author
\author[Kristian Seip]{Kristian Seip}
\address{Department of Mathematical Sciences, Norwegian University of Science and Technology,
NO-7491 Trondheim, Norway} \email{seip@math.ntnu.no}
\thanks{The author is supported by the Research Council of Norway
grant 160192/V30.}

%\keywords{Dirichlet series, boundary behaviour}
%\date{November 14, 2007.}

% ----------------------------------------------------------------------

\begin{abstract}
The Dirichlet--Hardy space $\Ht$ consists of those Dirichlet series
$\sum_n a_n n^{-s}$ for which $\sum_n |a_n|^2<\infty$. It is shown
that the Blaschke condition in the half-plane $\operatorname{Re}
s>1/2$ is a necessary and sufficient condition for the existence of
a nontrivial function $f$ in $\Ht$ vanishing on a given bounded
sequence. The proof implies in fact a stronger result: every
function in the Hardy space $H^2$ of the half-plane
$\operatorname{Re} s>1/2$ can be interpolated by a function in $\Ht$
on such a Blaschke sequence. Analogous results are proved for the
Hilbert space $\Da$ of Dirichlet series $\sum_n a_n n^{-s}$ for
which $\sum_n |a_n|^2[d(n)]^\alpha <\infty$; here $d(n)$ is the
divisor function and $\alpha$ a positive parameter. In this case,
the zero sets are related locally to the zeros of functions in
weighted Dirichlet spaces of the half-plane $\operatorname{Re}
s>1/2$. Partial results are then obtained for the zeros of functions
in $\Hp$ ($L^p$ analogues of $\Ht$) for $2<p<\infty$, based on
certain contractive embeddings of $\Da$ in $\Hp$.

\end{abstract}

%\date{\today}
\maketitle

\section{Introduction}

This paper studies the zeros of functions in certain Hilbert spaces
of Dirichlet series. The prototypical case is that of the
Dirichlet--Hardy space
\[ \Ht=\left\{f(s)=\sum_{n=1}^\infty a_n n^{-s}:\
\|f\|_{\Ht}^2=\sum_{n=1}^\infty |a_n|^2<\infty\right\},\] which we
view as a space of analytic functions in the half-plane
$\C_{1/2}^+=\{s=\sigma+it:\ \sigma>1/2\}$. Our result regarding
$\Ht$ is as follows.
\begin{theorem}
Suppose $S=(\sigma_j+it_j)$ is a bounded sequence of points in
$\C_{1/2}^+$. Then there is a nontrivial function in $\Ht$ vanishing
on $S$ if and only if $\sum_j (\sigma_j-1/2)<\infty$.
\end{theorem}
Here multiple zeros are accounted for in the usual way: $n$
occurrences of some point in the sequence $S$ correspond to a zero
of order at least $n$ at that point. Note that the summability
condition on $S$ is just the Blaschke condition for bounded
sequences in $\C_{1/2}^+$.

The ``only if'' part of Theorem~1 is well-known \cite{HLS}; it is a
consequence of the fact that if $f$ is in $\Ht$, then $f(s)/s$ is in
$H^2(\C_{1/2}^+)$. The novelty of Theorem~1 is thus the positive
direction, i.e., the sufficiency of the local Blaschke condition.
This result may come as no surprise to readers familiar with recent
developments such as \cite{Ose} and \cite{OS}, which quantify the
local resemblance between $\Ht$ and the Hardy space
$H^2(\C_{1/2}^+)$. It should be kept in mind, however, that in the
present setting the usual one-variable tool of dividing out zeros
must be abandoned, and no direct use of Blaschke products is
possible. Moreover, as shown in \cite{HLS}, the multiplier algebra
of $\Ht$ is the much smaller space $\Hi$ consisting of those
functions in $\Ht$ that extend to bounded holomorphic functions in
the larger half-plane $\C^+=\{s=\sigma+it:\ \sigma>0\}$. These facts
point to a certain rigidity of our problem.

A simple argument using almost periodicity of the function $t\mapsto
f(\sigma+it)$ along with Rouch\'{e}'s theorem \cite{Ose} shows that
a function in $\Ht$ has either none or an infinite number of zeros.
It appears to be an inaccessible problem to describe the pattern of
``repetitions'' of zeros in the vertical direction, and therefore a
full description of the zero sets of functions in $\Ht$ does not
seem to be within reach.

The main ingredient in our construction of a nontrivial function in
$\Ht$ vanishing on a bounded Blaschke sequence is an iteration
involving approximations by Dirichlet series on compact sets and
solutions to certain $\overline{\partial}$ equations. We will
present this proof in the next section, where we also notice that
our argument implies a stronger result: The values of an arbitrary
function in $H^2(\C_{1/2}^+)$ can be interpolated by a function in
$\Ht$ on a bounded Blaschke sequence.

Suitably elaborated, the same technique leads to analogous results
for most of the Hilbert spaces studied in \cite{O}. In Section 3, we
choose to present these results for the Hilbert spaces
\[ \Da=\left\{f(s)=\sum_{n=1}^\infty a_n n^{-s}:\
\|f\|_{\Da}^2=\sum_{n=1}^\infty
|a_n|^2[d(n)]^{\alpha}<\infty\right\},\] where $d(n)$ is the divisor
function and $\alpha$ a real parameter; we define $\Din$ as the
subspace of $\Ht$ consisting of those $\sum_n a_n n^{-s}$ for which
$a_n=0$ unless $n=1$ or $n$ is a prime . We will restrict to the
case $0<\alpha\le \infty$. Our reason for doing so is that this
leads to nontrivial results for the zeros of functions in the
Dirichlet--Hardy spaces $\Hp$ for $p>2$, to be considered in Section
4.

\section{Proof of Theorem~1}

The Hardy space $H^2(\C_{1/2}^+)$ is defined as the set of functions
$f$ analytic in $\C_{1/2}^+$ for which
\[ \| f \|_{H^2(\C_{1/2}^+)}^2=\sup_{\sigma>1/2} \int_{-\infty}^{\infty} |f(\sigma+it)|^2 dt
<\infty. \] Every $f$ in $H^2(\C_{1/2}^+)$ has a nontangential
boundary limit at almost every point of the vertical line
$\sigma=1/2$, and the corresponding limit function $t\mapsto
f(1/2+it)$ is in $L^2(\RR)$; the $L^2$ norm of this function
coincides with the $H^2$ norm introduced above. We may represent $f$
as
\[ f(s)=\int_{0}^\infty \varphi(\xi) e^{-(s-1/2) \xi} d\xi \]
so that by the Plancherel identity $\| f
\|_{H^2(\C_{1/2}^+)}=\sqrt{2\pi} \|\varphi\|_2$. By an appropriate
discretization of the integral in this representation, we obtain an
approximation on compact sets of functions in $H^2(\C_{1/2}^2)$ by
functions in $\Ht$:

\begin{lemma} \label{lem0} Let $N$ be a positive integer. Then for every $\varphi$
in $L^2(\log N,\infty)$, there is a function $F(s)=\sum_{n=N}^\infty
a_n n^{-s}$ in $\Ht$ such that $\|F\|_{\Ht}\le \|\varphi\|_2$ and
the function
\[ \Phi(s)=\int_{\log N}^\infty \varphi(\xi) e^{-(s-1/2) \xi} d\xi - F(s) \]
enjoys the estimate
\[
|\Phi(s)|\le 2 |s-1/2| N^{-\sigma-1/2} \|\varphi\|_2\] for $s$ in
$\C^+_{1/2}$.
\end{lemma}

\begin{proof}
We set
\[ a_n=\sqrt{n} \int_{\log n}^{\log(n+1)} \varphi(\xi) d\xi \]
and perform straightforward estimates.
\end{proof}

Note that a version of Lemma~\ref{lem0} was used iteratively by
J.-F.~Olsen in \cite[pp. 109--112]{O9}, in a similar way as will be
done below, to find a new proof of the interpolation theorem for
$\Ht$ first proved in \cite{Ose}. The desired Dirichlet series will
in our case be an infinite series of functions in $\Ht$, where each
term is determined by an application of Lemma~\ref{lem0} and
convergence is ensured when $N$ is sufficiently large.

We turn to the second essential ingredient in our iteration. Set
\[ \Omega=\Omega(R, \tau)=\{s=\sigma+it:\ 1/2 \le \sigma\le 1/2 +\tau, \ -R \le t \le R\},
\]
for positive numbers $R$ and $\tau$. Lebesgue area measure on $\C$
is denoted by $m$. The following lemma has an obvious proof which is
omitted.
\begin{lemma}
Suppose $g$ is a continuous function on $\C_{1/2}^+$ supported on
$\Omega=\Omega(R, 2)$ and satisfying $|g(s)|\le \epsilon$. Then
\[ u(s)=\frac{1}{\pi}\int_{\Omega} \frac{g(w)}{s-w} dm(w) \]
solves $\overline{\partial}u=g$ in $\C_{1/2}^+$ with bounds $ \|
u\|_\infty \le c \epsilon \log R $ for an absolute constant $c$
(independent of $R$) and
\[ |u(s)|\le \frac{R \epsilon}{\pi \operatorname{dist}(s,\Omega)}. \]
\label{lem1}
\end{lemma}
It follows that \[ \sup_{\sigma\ge 1/2} \left(\int_{-\infty}^\infty
|u(\sigma+it)|^2 dt\right)^{1/2} \le c' \epsilon \sqrt{R} \log R \]
for an absolute constant $c'$.

Let now $S=(\sigma_j+it_j)$ be a bounded Blaschke sequence in
$\C_{1/2}^+$ with an associated Blaschke product $B$. We may assume
that $S$ is in $\Omega(R-2, 1/2)$ for some $R>2$. We fix once and
for all a smooth function $\Theta$ on the closed half-plane $\sigma\ge 1/2$ with the following properties: $\Theta$  is supported on $\Omega(R,2)$ such
that $\Theta(s)=1$ for $s$ in $\Omega(R-1,1)$ and $|\nabla
\Theta|\le 2$. For a given positive integer $N$, we set
$E_N(s)=N^{-s+1/2}$ and define an operator $T_N$ on $E_N
H^2(\C_{1/2}^+)$ as follows. Set
\[ f(s)=\int_{\log N}^\infty \varphi(\xi) e^{-(s-1/2)\xi} d\xi \]
and $\Phi=f-F$, where $F$ is as in Lemma~\ref{lem0}, and let $u$
denote the solution from Lemma~\ref{lem1} to the equation
\[ \overline{\partial} u = \frac{\overline{\partial}(\Theta \Phi)}{B E_N}.\]
Then set
\[ T_Nf=\Theta \Phi-BE_N u.\]
It is clear that $T_Nf$ is in $E_N H^2(\C_{1/2}^+)$ since $\Theta$
has compact support. The virtue of $T_N$ is that $T_N f(s)= \Phi(s)$
for $s$ in $S$, i.e., $T_Nf-\Phi$ is divisible by $B$.

The following estimate is crucial.
\begin{lemma}\label{crucial1}
The operator $T_N$ on $E_N H^2(\C_{1/2}^+)$ enjoys the estimate $\|
T_N\|\le C N^{-1},$ where $C$ is a constant depending only on $R$
and $S$.
\end{lemma}

\begin{proof} By our assumption on $S$, $|B(s)|$ is bounded below on
the set where $\nabla \Theta\neq 0$. We therefore obtain the result
by combining the estimates of Lemma~\ref{lem0} and Lemma~\ref{lem1}.
\end{proof}

\begin{proof}[Final part of the proof of Theorem~1] Set
$f_0(s)=B(s)E_N(s)/(s+1/2)$, where $N$ is a sufficiently large
integer. Then $f_0$ is in $E_N H^2(\C_{1/2}^+)$ and its
$H^2(\C_{1/2}^+)$ norm is $\sqrt{2\pi}$. Set $f_j=T_N^j f_0$, and
let $F_j$ be the Dirichlet series in $\Ht$ obtained by applying
Lemma~\ref{lem0} to $f_j$. Then $F_0+f_1$ vanishes on $S$ since
$f_1=f_0-F_0$ on $S$. Iterating, we get that $F_0+\cdots + F_j+
f_{j+1}$ also vanishes on $S$. Thus the function
\begin{equation}\label{FF} F=\sum_{j=0}^\infty F_j \end{equation}
is a nontrivial function in $\Ht$ vanishing on $S$ if we choose $N$
so large that $\|T_N\|<1$ and, say,
\[ |F_0(3/2)|>\sum_{j=1}^\infty |F_j(3/2)|.\]
Both inequalities can be achieved in view of respectively
Lemma~\ref{crucial1} and Lemma~\ref{lem0}.
\end{proof}
It is of interest to note that our algorithm can be applied in the
following way. Let $f$ be an arbitrary function in $H^2(\C_{1/2}^+)$
that is not divisible by the Blaschke product $B$. We may set
$f_0=f$ and apply $T_N$ to $f$ as before, now viewing $T_N$ as an
operator from $H^2(\C_{1/2}^+)$ to $E_N H^2(\C_{1/2}^+)$. In the
first step of the iteration, we get a poorer bound (depending on
$N$), but Lemma~\ref{crucial1} applies when $T_N$ acts on $f_j$ for
$j\ge 1$. The proof shows that $F$ agrees with $f$ on $S$, which we
again take to mean that $f-F$ is divisible by $B$. What we have
proved, can be summarized as follows.

\begin{theorem1p}
Suppose $S=(\sigma_j+it_j)$ is a bounded Blaschke sequence in
$\C_{1/2}^+$. Then for every function $f$ in $H^2(\C^+_{1/2})$ there
is a Dirichlet series $F$ in $\Ht$ that agrees with $f$ on $S$.
\end{theorem1p}

This theorem can be viewed as a precise statement about the local
resemblance between $H^2(\C^+_{1/2})$ and $\Ht$ but of a different
nature than the approximation theorem in \cite{OS}. Namely, since
the function $f-F$ is divisible by a Blaschke product, it can not be
continued analytically across any segment containing an accumulation
point of the Blaschke sequence.

On the other hand, we may relate our discussion to the results of
\cite{Ose} and \cite{OS}. First, note that Theorem~1' implies the
local interpolation theorem for $\Ht$ presented in \cite{Ose}.
Second, if we replace $B$ by the constant $1$ in the proof of
Theorem~1', we get that the function $f-F$ does indeed extend to an
analytic function in $\C\setminus (1/2+i(-R+1,R-1))$; here $f$ is an arbitrary function in 
$H^2(\C^+_{1/2})$ and $F$ is the function in \eqref{FF} obtained by starting the iteration with $f_0=f$. This result may
be compared to Theorem~1 in \cite{OS}.

\section{Zeros of functions in $\Da$}

In \cite{O}, Olsen studied the relation between $\Da$ and the
ordinary weighted Dirchlet spaces in the half-plane $\C^+_{1/2}$. We
will use the link found by Olsen to obtain local results on the
zeros of functions in $\Da$.

For $0\le \beta \le 1$, we let $D_{\beta}(\C_{1/2}^+)$ consist of
those functions $f$ in $H^2(\C_{1/2}^+)$ for which
\[ \int_{\C_{1/2}^+}|f'(s)|^2(\sigma-1/2)^{1-\beta} dm(s)<\infty. \]
Since $f$ belongs to $H^2(\C_{1/2}^+)$, we may write
\[ f(s)=\int_{0}^{\infty} \varphi(\xi) e^{-(s-1/2)\xi} d\xi \]
for some function $\varphi$ in $L^2(0,\infty)$. A computation
involving the Plancherel identity shows that
\[ \int_{\C_{1/2}^+}|f'(s)|^2(\sigma-1/2)^{1-\beta}
dm(s)=c\int_0^\infty |\varphi(\xi)|^2 \xi^\beta d\xi\] for an
absolute constant $c$. We may therefore equip
$D_{\beta}(\C_{1/2}^+)$ with the norm
\[ \|f\|_{D_\beta(\C_{1/2}^+)}=\left(\int_{0}^\infty |\varphi(\xi)|^2 (1+\xi^\beta)
d\xi\right)^{1/2}.\]

It will be convenient to use the following terminology. Given a
class of analytic functions $\Cl$ on some domain $\Omega$ in $\C$,
we say that a sequence $\Lambda$ of not necessarily distinct points
in $\Omega$ belongs to $Z(\Cl)$ if there is a nontrivial function in
$\Cl$ having $\Lambda$ as its zero set. We will also a few times
resort to the notation $U\lesssim V$ which means that there is a
positive constant $C$ such that $U\le CV$ holds for whatever
arguments the quantities $U$ and $V$ depend on.

\begin{theorem}\label{two}
Suppose $0<\alpha\le \infty$ and let $S=(\sigma_j+it_j)$ be a
bounded sequence in $\C_{1/2}^+$. Then there is a nontrivial
function in $\Da$ vanishing on $S$ if and only if $S$ belongs to
$Z(D_{1-2^{-\alpha}}(\C_{1/2}^+))$.
\end{theorem}

The ``only if'' part of this theorem follows from Theorem~1 of
\cite{O} and will therefore not be considered in what follows.

A number of partial results regarding $Z(D_{\beta}(\C_{1/2}^+))$ for
$0<\beta \le 1$ are known thanks to work of L. Carleson, H. Shapiro
and A. Shields, and others \cite{C, C2, SS, AM, MS, PP}. These
papers deal with $D_{\beta}(\D)$, i.e., the space of analytic
functions $f$ on the unit disk $\D$ for which
\[ \int_{\D}|f'(z)|^2(1-|z|^2)^{1-\beta} dm(z)<\infty; \]
the results regarding $D_{\beta}(\D)$ apply because of the following
fact: If $\phi$ is a conformal map of $\C_{1/2}^+$ onto the unit
disk $\D$, then $S$ is in $Z(D_{\beta}(\C_{1/2}^+))$ if and only if
$\phi(S)$ is in $Z(D_{\beta}(\D))$. Indeed, if $f$ is in
$D_{\beta}(\D)$ and we choose
\[ \phi(s)=\frac{s-3/2}{s+1/2}, \]
then a calculation shows that the function
\[ F(s)=f(\phi(s)) (s+1/2)^{\beta-2} \]
is in $D_{\beta}(\C_{1/2}^+)$; a similar transformation can be found
for the reverse inclusion.

One of Carleson's theorems \cite{C} implies that
\[ \sum_j (\sigma_j-1/2)^{1-\beta} <\infty \]
is a sufficient condition for a bounded sequence $S=(\sigma_j+it_j)$
to be in $Z(D_{\beta}(\C_{1/2}^+))$ for $0<\beta<1$; a theorem of
Shapiro and Shields \cite{SS} shows that
\[ \sum_j |\log(\sigma_j-1/2)|^{-1}<\infty \]
suffices for $S$ to be in $Z(D_{1}(\C_{1/2}^+))$. If $S=(\sigma_j+i
t_j)$ is restricted to a cone $|t-t_0|\le c (\sigma-1/2)$, then the
Blaschke condition is sufficient for the sequence to belong to
$Z(D_{1}(\C_{1/2}^+))$, as follows from an observation in \cite{C2}.
For information about further developments, we refer to \cite{PP}
and \cite{MS}.

Of particular interest to us is the following result.

\begin{lemma} \label{subsetanal}
Assume $0<\beta\le 1$. If $S=(\sigma_j+it_j)$ belongs to
$Z(D_{\beta}(\C_{1/2}^+))$, then there exists a function $G$ in
$D_{\beta}(\C_{1/2}^+)\cap H^\infty(\C_{1/2}^+)$ that has $S$ as its
zero set and that can be continued analytically to the domain
\[ \C\setminus \bigcup_j ((-\infty,1-\sigma_j]+i t_j).\]
\end{lemma}

\begin{proof} The lemma is a combination of Proposition 9.33 and Proposition 9.37 in
\cite[p. 137]{AM} and Theorem 2 in \cite{AC}. Here we use the fact
that the multiplier algebra of $D_{\beta}(\C_{1/2}^+)$ is a subset
of $H^\infty(\C_{1/2}^+)$.\end{proof}

Note that the only property that we will need regarding the analytic continuation of the function $G$ in Lemma~\ref{subsetanal}, is that $G$ is uniformly bounded with bounded derivative 
on any subset of the closed half-plane $\sigma\ge 1/2$ which is at a positive distance from $S$.

We set
\[ \|\varphi\|_{2,\beta}^2=\int_{R}^\infty |\varphi(\xi)|^2 \xi^{\beta}
d\xi,\] assuming $\varphi$ is a function defined for $\xi\ge R>0$,
and let $L^2_{\beta}(R,\infty)$ be the set of measurable functions
$\varphi$ on the half-line $[R,\infty)$ for which
$\|\varphi\|_{2,\beta}<\infty$. The following lemma is analogous to
Lemma~\ref{lem0}.
\begin{lemma} \label{lemD} Let $N$ be a positive integer and $K$ a compact
subset of the closed half-plane $\overline{\C_{1/2}^+}$, and assume
that $0<\alpha\le \infty$. Then there exist constants $c, C, C_K$
and $\eta>1/2$, $\nu>1$ such that, for every $\varphi$ in
$L^2_{1-2^{-\alpha}}(\log N,\infty)$, there is a function
$F(s)=\sum_{n=N}^\infty a_n n^{-s}$ in $\Da$ satisfying
$\|F\|_{\Da}\le c \|\varphi\|_{2,1-2^{-\alpha}}$ and such that the
function
\[\Phi(s)=\int_{\log N}^\infty \varphi(\xi) e^{-(s-1/2)\xi}
d\xi - F(s) \] enjoys the estimates \begin{equation} |\Phi(s)|  \le
C |s-1/2| N^{-\sigma+1/2} (\log N)^{-\eta}
\|\varphi\|_{2,1-2^{-\alpha}} \label{point} \end{equation} for $s$
in $\C_{1/2}^{+}$ and
\begin{equation}
 \int_{K} |\Phi'(s)|^2 (\sigma -1/2)^{2^{-\alpha}} dm(s)  \le
C_K (\log N)^{-\nu} \| \varphi\|_{2,1-2^{-\alpha}}^2.
\label{derivative}
\end{equation}
\end{lemma}

We get a poorer bound in \eqref{point} than in Lemma~\ref{lem0}
because asymptotic estimates for the divisor function are involved
in the construction of the Dirichlet series $F$.

\begin{lemma}\label{RamWil}
If $0<\alpha\le \infty$, then we have
\[\sum_{ n\le M} [d(n)]^{-\alpha}
= A M (\log M)^{2^{-\alpha}-1}(1+O((\log M )^{-1}))
\] for some absolute constant $A>0$ when $M\to\infty$.
\end{lemma}

The formula in the lemma (in a more precise form) was stated by
Ramanujan \cite[Formula (9)]{R} and later proved by Wilson in
\cite[Formula (2.39)]{W}. For $\alpha=\infty$, the estimate follows
from the prime number theorem.

\begin{proof}[Proof of Lemma~\ref{lemD}]
Let $n_j$ be the smallest integer $n$ such that $e^{j^{\gamma}}\le
n$. We set $\xi_{n_j}=j^{\gamma}$; when $e^{j^\gamma}<n\le
e^{(j+1)^\gamma}$, we choose $\xi_n$ inductively such that
\begin{equation}\label{length} \xi_{n+1}-\xi_n=A_j
[d(n)]^{-\alpha}e^{-j^\gamma}j^{\gamma(1-2^{-\alpha})},\end{equation}
where $A_j$ is chosen such that $\xi_{n_{j+1}}=(j+1)^\gamma$. Note
if $1/2<\gamma<1$, then Lemma~\ref{RamWil} implies that
\begin{equation}\sum_{j^{\gamma}\le \log n\le (j+1)^\gamma} [d(n)]^{-\alpha} =
A e^{j^\gamma}j^{\gamma-1}j^{\gamma(2^{-\alpha}-1)}(1+o(1))
\label{RW}
\end{equation} for some absolute constant $A>0$ when $j\to\infty$.
We will therefore assume that $1/2<\gamma <1$.

We set
\[ a_n=\sqrt{n}\int_{\xi_n}^{\xi_{n+1}} \varphi(\xi) d\xi\]
and note first that
\[ |a_n|^2\le n (\xi_{n+1}-\xi_n) \int_{\xi_n}^{\xi_{n+1}}|\varphi(\xi)|^2d\xi
\] by the Cauchy--Schwarz inequality. Taking into account
\eqref{length}, we obtain the norm estimate $\|F\|_{\Da}\le c
\|\varphi\|_{2,1-2^{-\alpha}}$ from this inequality.

To obtain the pointwise estimate \eqref{point}, we consider first
the finite sum
\[ \Sigma_j=\sum_{n=n_j}^{n_{j+1}-1} \left (\int_{\xi_n}^{\xi_{n+1}}
\varphi (\xi)e^{-(s-1/2)\xi} d\xi - a_n n^{-s} \right).\] We observe
that we have
\[ |n^{-s+1/2}-e^{-(s-1/2)\xi}|\le N^{-\sigma+1/2} |s-1/2|
j^{\gamma-1} \] as long as $\xi$ and $\xi_n$ both are in the
interval $[\xi_{n_j},\xi_{n_{j+1}}]$. Applying again the
Cauchy--Schwarz inequality and taking into account \eqref{length},
we therefore obtain
\[ \Big|\int_{\xi_n}^{\xi_{n+1}}
\varphi(\xi)e^{-(s-1/2)\xi} d\xi - a_n n^{-s}\Big|^2\ \ \ \ \ \ \ \
\ \ \ \ \ \ \ \ \ \ \ \ \ \ \ \ \ \ \ \ \ \ \ \ \ \ \ \ \ \ \ \ \ \
\ \ \ \ \ \ \ \ \ \ \ \ \]
\[
\ \ \ \ \ \ \ \ \ \ \ \ \ \ \ \le A_j N^{-2\sigma+1}|s-1/2|^2
j^{2(\gamma-1)} [d(n)]^{-\alpha} n^{-1} (\log n)^{1-2^{-\alpha}}
\int_{\xi_n}^{\xi_{n+1}}|\varphi(\xi)|^2d\xi.\] Now applying the
Cauchy--Schwarz inequality to the sum, we get
\[ |\Sigma_j|^2 \le A_j
N^{-2\sigma+1}|s-1/2|^2 j^{2(\gamma-1)}\sum_{n=n_j}^{n_{j+1}}
[d(n)]^{-\alpha} n^{-1} (\log n)^{1-2^{-\alpha}}
\int_{\xi_{n_j}}^{\xi_{n_{j+1}}}|\varphi(\xi)|^2d\xi. \] By
\eqref{RW}, we therefore get
\[
|\Sigma_j|^2\lesssim N^{-2\sigma+1}|s-1/2|^2
j^{3(\gamma-1)-\gamma(1-2^{-\alpha})}
\int_{\xi_{n_j}}^{\xi_{n_{j+1}}} \xi^{1-2^{-\alpha}}
|\varphi(\xi)|^2 d\xi.\] If we now choose a sufficiently small
$\gamma$ in the interval $(1/2, 1)$, then the series
\[ \sum_{j=1}^\infty j^{-3+\gamma(2+2^{-\alpha})} \]
is summable and we get \eqref{point} with
\[ \eta=(2-\gamma(2+2^{-\alpha}))/\gamma.\]

To deal with \eqref{derivative}, we begin by assuming that $K\subset
\Omega(R,\tau)$ for suitable $R$ and $\tau$. By duality, we have
\[ \left(\int_{-R}^R |\Phi'(\sigma+it)|^2 dt\right)^{1/2}=
\sup_{\|g\|_2=1} \left|\int_{-R}^R \Phi'(\sigma+it) g(t) dt\right|,
\] where the supremum is taken over all $g$ of norm 1 in
$L^2(-R,R)$. We obtain
\begin{eqnarray*}
 \int_{-R}^R \Phi'(\sigma+it) g(t) dt &= & \sum_{n=N}^\infty
\left(a_n (\log n) n^{-\sigma} \hat{g}(\log
n)-\int_{\xi_n}^{\xi_{n+1}} \varphi(\xi) \xi
e^{-(\sigma-1/2)\xi}\hat{g}(\xi) d\xi\right) \\
&=& \sum_{n=N}^\infty \int_{\xi_n}^{\xi_{n+1}}
  \varphi(\xi) ((\log n)n^{-\sigma+1/2}\hat{g}
(\log n)-\xi e^{-(\sigma-1/2)\xi}\hat{g}(\xi) )d\xi.
\end{eqnarray*}
It follows that
\[ \left|\int_{-R}^R \Phi'(\sigma+it) g(t) dt\right|
\lesssim (\log N)^{1-1/\gamma}
\sum_{j}\sum_{n=n_j}^{n_{j+1}-1}|a_n|(\log n)
n^{-\sigma}(|\hat{g}(\xi_j^*)|+|\hat{g}'(\xi_j^{**})|), \] where
$|\hat{g}(\xi_j^*)|$ and $|\hat{g}'(\xi_j^{**})|$ are the maxima of
the respective functions $|\hat{g}(\xi)|$ and $|\hat{g}'(\xi)|$ on
$[\xi_{n_j},\xi_{n_{j+1}}]$. We apply the Cauchy--Schwarz inequality
and employ again Lemma~\ref{RamWil}; we also use the
Plancherel--P\'{o}lya inequality (cf. the proof on pp. 2674--2675 in
\cite{O}) and get
\[ \int_{-R}^R |\Phi'(\sigma+it)|^2 dt
\lesssim (\log N)^{2-2/\gamma} \sum_{n\ge N} |a_n|^2 [d(n)]^{\alpha}
(\log n)^{2^{-\alpha}+1} n^{-2\sigma+1},\] where the implicit
constant depends on $R$. Since
\[ \int_{1/2}^\infty (\log n)^{2^{-\alpha}+1}
n^{-2\sigma+1}(\sigma-1/2)^{2^{-\alpha}}d\sigma=2^{-1-2^{-\alpha}}\Gamma(1+2^{-\alpha}),\]
\eqref{derivative} follows with $\nu=2/\gamma-2$, where as in the
preceding case we get the desired result by choosing a sufficiently
small $\gamma$ in the interval $(1/2,1)$.
\end{proof}

Let now $S=(\sigma_j+it_j)$ be a bounded sequence in
$Z(D_{1-2^{-\alpha}}(\C_{1/2}^+))$, and let $G$ be the function from
Lemma~\ref{subsetanal} vanishing on $S$. We may again assume that
$S$ is in $\Omega(R-2, 1/2)$ for some $R>2$ and let $\Theta$ be as
above. We view $E_N D_{1-2^{-\alpha}}(\C_{1/2}^+)$ as a Hilbert
subspace of $D_{1-2^{-\alpha}}(\C_{1/2}^+)$ and define $T_N$ on $E_N
D_{1-2^{-\alpha}}(\C_{1/2}^+)$ similarly as in the preceding
section: Set
\[ f(s)=\int_{\log N}^\infty \varphi(\xi) e^{-(s-1/2)\xi} d\xi \]
and $\Phi=f-F$, where $F$ is as in Lemma~\ref{lemD}, and let $u$
denote the solution from Lemma~\ref{lem1} to the equation
\[ \overline{\partial} u = \frac{\overline{\partial}(\Theta \Phi)}{G E_N}.\]
Then set
\[ T_Nf=\Theta \Phi-GE_N u.\]
We will again use that $T_N f(s)= \Phi(s)$ for $s$ in $S$. 
\begin{lemma}\label{crucial2}
The operator $T_N$ acts boundedly on $E_N D_{1-2^{-\alpha}}(\C_{1/2}^+)$
with $\| T_N\|\le C (\log N)^{-\delta},$ where $\delta>0$ and $C$ is
a constant depending only on $R$ and $S$.
\end{lemma}

\begin{proof} We set for convenience $\beta=1-2^{-\alpha}$. It is clear that $T_N f/E_N$ is in $H^2(\C_{1/2}^+)$ so it is enough to show
that
\[ \int_{\C_{1/2}^+}\left|\partial (T_N f)(s)\right|^2(\sigma-1/2)^{1-\beta} dm(z) \le
C (\log N)^{-\delta} \|f\|^2_{D_\beta(\C_{1/2}^+)}.\]  We observe
that \eqref{derivative} implies that
\[ \int_{\C_{1/2}^+}|\partial (\Theta \Phi)(s)|^2(\sigma-1/2)^{1-\beta} dm(z) \lesssim
(\log N)^{-\nu} \|f\|^2_{D_\beta(\C_{1/2}^+)},\] and so it remains
to consider the weighted area integral for the function
\begin{equation}\label{three}
\partial (GE_N u)=\partial G E_N u + G \partial E_N u + G E_N
\partial u.\end{equation}
We treat each of the terms on the right-hand side of \eqref{three}
separately. We get \[ \int_{\C_{1/2}^+}|\partial G (s)
E_N(s)u(s)|^2(\sigma-1/2)^{1-\beta} dm(z) \lesssim (\log N)^{-2\eta}
\|f\|^2_{D_\beta(\C_{1/2}^+)}\] by our assumption on $G$ and the
estimate on $u$ obtained from \eqref{point} and Lemma~\ref{lem1}. As
for the second term on the right-hand side of \eqref{three}, we
observe that
\[ |\partial E_N(s) u(s)|\lesssim (\log N)^{1-\eta} N^{-\sigma+1/2} (1+|s|)^{-1} \|f\|_{D_\beta(\C_{1/2}^+)}\]
by \eqref{point} and Lemma~\ref{lem1}. Hence, using also
Lemma~\ref{subsetanal}, we obtain
\[ \int_{\C_{1/2}^+}| G (s)
\partial E_N(s)u(s)|^2(\sigma-1/2)^{1-\beta} dm(z) \lesssim (\log
N)^{\beta-2\eta} \|f\|^2_{D_\beta(\C_{1/2}^+)},\] where
$2\eta-\beta>0$ because $\eta>1/2$.

It remains to estimate the contribution from the third term on the
right-hand side of \eqref{three}. Set $\Delta =\Omega(R,2)\setminus
\Omega(R-1,1)\supseteq \operatorname{supp}(\nabla \Theta)$. Then
trivially
\[ |\partial u(s)| \lesssim (\log N)^{-\eta} (1+|s|)^{-2}\] if
$\operatorname{dist}(s,\Delta)\ge 1/8$. On the other hand, if
$\operatorname{dist}(s,\Delta)< 1/8$, we argue as follows. Set
\[ \Delta_+=\Omega(R+1,3)\setminus \Omega(R-3/2,3/4)\] and apply the
Cauchy--Pompeiu formula to $\Theta \Phi/(GE_N)$ in $\Delta_+$. Hence
we get
\[ \partial u(s) = \partial \left(\frac{\Theta \Phi}{G
E_N}\right)(s) - \frac{1}{2\pi i}\int_{\partial \Delta_+}
\frac{\Theta(w) \Phi(w)}{G (w) E_N (w)}\frac{1}{(s-w)^2} dw.\] 
We write $\Theta(w)=\Theta (s )+\Theta(w)-\Theta
(s)$. Then, by  the smoothness of $\Theta$ and analyticity of $\Phi/(GE_N)$,  we
get 
\[ \left|\int_{\partial \Delta_+}
\frac{\Theta(w) \Phi(w)}{G (w) E_N (w)}\frac{1}{(s-w)^2} dw\right|\lesssim 
|\Theta(s)|  \left|\partial \left(\frac{ \Phi}{G
E_N}\right)(s)\right| + \int_{\partial \Delta_+}
\frac{ |\Phi(w)|}{|G (w) E_N (w)|}\frac{1}{|s-w|} |dw|.\]
Using Lemma~\ref{subsetanal} and \eqref{point}, we therefore deduce that
\[ |\partial u(s)|\lesssim \left|\partial \left(\frac{\Theta \Phi}{G
E_N}\right)(s)\right|+(\log N)^{-\eta} |\log (\sigma-1/2)| \|f\|_{D_\beta(\C_{1/2}^+)}\] 
whenever $\operatorname{dist}(s,\Delta)< 1/8$.
The
remaining part of the estimation is essentially a repetition of the
preceding calculations, and the details are therefore omitted.
\end{proof}

\begin{proof}[Final part of the proof of Theorem~\ref{two}] We act in the same way as in
the proof of Theorem~1. Set $f_0=E_N G$, where $N$ is a sufficiently
large integer, and then $f_j=T_N^j f_0$. Let $F_j$ be the Dirichlet
series in $\Da$ obtained from $f_j$. Then, by the same argument as
in the proof of Theorem~1, $F=\sum_{j=0}^\infty F_j$ is a nontrivial
function in $\Da$ vanishing on $S$ if $N$ is chosen so large that
$\|T_N\|<1$ and, say,
\[ |F_0(3/2)|>\sum_{j=1}^\infty |F_j(3/2)|.\]
This time we use respectively Lemma~\ref{crucial2} and
Lemma~\ref{lemD} to conclude that these inequalities are fulfilled
when $N$ is sufficiently large.
\end{proof}

In accordance with the observation leading to Theorem~1' in the
preceding section, we get that our algorithm in fact yields the
following stronger result.

\begin{theorem2p}
Suppose $S=(\sigma_j+it_j)$ is a bounded sequence in
$Z(D_{1-2^{-\alpha}}(\C_{1/2}^+))$. Then for every function $f$ in
$D_{1-2^{-\alpha}}(\C^+_{1/2})$ there is a Dirichlet series in $\Da$
that agrees with $f$ on $S$.
\end{theorem2p}

\section{The relation between $\Da$ and $\Hp$}

For $1\le p < \infty$, we define $\Hp$ as the closure of the set of
finite Dirichlet polynomials in the norm
\[ \lim_{T\to\infty} \left(\frac{1}{T} \left|\sum_{n} a_n n^{-it}\right|^p dt \right)^{1/p}. \]
Alternatively, we may express this limit as an $L^p$ norm over the
infinite-dimensional torus $\T^\infty$. We refer to F. Bayart's
paper \cite{Ba}, where these spaces were first studied.

We will only consider the case $2\le p <\infty$. Since then $\Hp$ is
a subset of $\Ht$, the Blaschke condition for bounded sequences in
$\C_{1/2}^+$ is a necessary condition for the existence of a
nontrivial Dirichlet series $f$ in $\Hp$ vanishing on a bounded
sequence $S$ in $\C_{1/2}^+$. We are not able to determine whether,
in general, this condition is sufficient as well, but the following
special case can be settled at once. Note first that Khinchin's
inequality implies that $\Din$ is included in $\Hp$ for every finite
$p$. (This fact will also follow from the sharper results to be
established below.) By Theorem~\ref{two} and Carleson's observation
from \cite{C2} mentioned above, this inclusion leads to the result
that for $1\le p <\infty$ the Blaschke condition is necessary and
sufficient for there to be a nontrivial function in $\Hp$ vanishing
on a bounded sequence contained in a cone $|t-t_0|\le
c|\sigma-1/2|$. It is of interest to note that this result fails
spectacularly when $p=\infty$ because $\Hi$ is a space of functions
analytic in the larger half-plane $\C^+$.\footnote{The Blaschke
condition $\sum_j \sigma_j <\infty$ is trivially a necessary and
sufficient condition for the existence of a nontrivial function in
$\Hi$ vanishing on a bounded sequence $(\sigma_j+it_j)$ in $\C^+$;
it suffices to observe that $(2^{-\sigma_j-it_j})$ is a Blaschke
sequence in $\D$, and so we may pick an associated Blaschke product
$B$ and use the function $B(2^{-s})$.}

We will now establish a more precise relation between $\Da$ and
$\Hp$.
\begin{lemma}\label{integer}
For $\alpha$ a nonnegative integer, the space $\Da$ is contractively
embedded in $\Ha$.
\end{lemma}
\begin{proof}
We prove the lemma by induction on $\alpha$. The statement is a
tautology when $\alpha=0$. Assume it holds for $\alpha=k$. This
means that if we write $g=f^2$ with $g(s)=\sum_n b_n n^{-s}$, then
we have
\[ \|f\|_{2^{k+1}}^4= \|g\|_{2^k}^{2} \le \sum_n |b_n|^2 [d(n)]^k.\]
Writing $f(s)=\sum_n a_n n^{-s}$ and using the Cauchy--Schwarz
inequality, we obtain
\[ |b_n|^2\le d(n) \sum_{k|n} |a_k|^2 |a_{n/k}|^2.\]
The result follows since $d(lk)\le d(l) d(k)$.
\end{proof}

We denote by $[\alpha]$ the integer part of $\alpha$ and obtain from
Lemma~\ref{integer} the following result for general $\alpha>0$.
\begin{theorem}
For $\alpha>0$, the space $\Da$ is contractively embedded in $\Hab$.
\end{theorem}
\begin{proof}
We use an argument similar to the one used in the proof of the
Riesz--Thorin theorem. Suppose
\[ f_{\alpha}(s)= \sum_n a_n n^{-s} \]
has norm 1 in $\Da$. Then the function
\[ f_z (s)= \sum_n a_n [d(n)]^{\alpha/2 - z/2}  n^{-s}\]
has norm bounded by 1 in $\Dx$, where $z=x+iy$. We take the inner
product between $f_z$ and $g |g|^{az+b-1}$ for an arbitrary function
$g$ on $\T^\infty$ taking finitely many values, where the parameters
$a$ and $b$ are chosen such that \[ a [\alpha] + b
=1-2^{-([\alpha]+1)} \ \ \ \text{and} \ \ \ a ([\alpha]+1) + b
=1-2^{-([\alpha]+2)}.\] This means that \[ a=2^{-([\alpha]+2)} \ \ \
\text{and} \ \ \ b=1-2^{-([\alpha]+2)}([\alpha]+2).\] We now
conclude from the preceding lemma and the three lines lemma that
$\|f_\alpha\|_p\le 1 $ when
\[ 1-1/p=a \alpha + b, \]
or in other words when $p=2^{[\alpha]+2}(2+[\alpha]-\alpha)^{-1}$.
\end{proof}

For $p=2^{k+1}$ and $k$ a positive integer, we get that if $S$ is a
sequence of bounded numbers belonging to $Z(D_{1-2/p}(\C^+_{1/2}))$,
then there is a nontrivial function in $\Hp$ vanishing on $S$. For
general $p$, the criterion becomes a bit more cumbersome, and we
state it explicitly only for $2<p<4$: If a bounded sequence $S$ is
in $Z(D_{1-2^{-2+4/p}}(\C^+_{1/2}))$, then there is a nontrivial
function in $\Hp$ vanishing on $S$.

It may be objected that we have found a rather indirect route to
arrive at these results for $\Hp$. Our method of proof reflects
that, in general, it remains a challenge to develop techniques that
would allow a more direct approach to the study of $\Hp$.

\vspace{2mm}

\noindent \textbf{Acknowledgements.} I am grateful to Jan-Fredrik
Olsen, Joaquim Ortega-Cerd\`{a}, Eero Saksman, and Jordi Pau for
helpful comments on the subject matter of this paper.


\begin{thebibliography}{BRSHZE}

\bibitem{AM} J. Agler and J.~E.~McCarthy, \emph{Pick Interpolation
and Hilbert Function Spaces}, Graduate Studies in Mathematics
\textbf{44}, Amer. Math. Soc., Providence RI, 2002.

\bibitem{AC} J.~E.~Akutowicz and L.~Carleson, \emph{The analytic
continuation og interpolatory functions}, J. Analyse Math.
\textbf{7} (1959/1960), 223–-247.

\bibitem{Ba} F. Bayart, \emph{Hardy spaces of Dirichlet series and
their composition operators}, Monatsh. Math. \textbf{136} (2002),
203--236.

\bibitem{C} L. Carleson, \emph{On a Class of Meromorphic Functions
and Its Exceptional Sets}, Thesis, University of Uppsala, 1950.

\bibitem{C2} L. Carleson, \emph{On the zeros of functions with
bounded Dirichlet integrals}, Math. Z. \textbf{56} (1952), 289--295.

\bibitem{HLS} H. Hedenmalm, P. Lindqvist, and K. Seip, \emph{A Hilbert
space of Dirichlet series and systems of dilated functions in
$L^2(0,1)$}, Duke Math. J. \textbf{86} (1997), 1–-37.

\bibitem{MS} J. Mashreghi and M. Shabankah, \emph{Admissible
functions for the Dirichlet space}, Studia Math. \textbf{198}
(2010), 147--156.

\bibitem{O9} J.-F. Olsen, \emph{Boundary Properties of Modified Zeta
Functions and Function Spaces of Dirichlet series}, Doctoral Thesis,
Norwegian University of Science of Technology, 2009.

\bibitem{O} J.-F. Olsen, \emph{Local properties of Hilbert spaces of
Dirichlet spaces}, J. Funct. Anal. \textbf{261} (2011), 2669--2696.

\bibitem{OS} J.-F. Olsen and E. Saksman,  \emph{On the boundary
behaviour of the Hardy spaces of Dirichlet series and a frame bound
estimate}, J. Reine Angew. Math. \textbf{663} (2012), 33--66.

\bibitem{Ose} J.-F.~Olsen and K. Seip, \emph{Local interpolation in Hilbert
spaces of Dirichlet series}, Proc. Amer. Math. Soc. \textbf{136}
(2008), 203-–212.

\bibitem{PP}J.~Pau and J.~A.~Pel\'{a}ez, \emph{On the zeros of functions in
Dirichlet-type spaces}, Trans. Amer. Math. Soc. \textbf{363} (2011),
1981-–2002.

\bibitem{R} S. Ramanujan, \emph{Some formulae in the analytic theory of
numbers}, Messenger Math. \textbf{45} (1916), 81--84.

\bibitem{SS} H.~S.~Shapiro and A.~L.~Shields, \emph{On the zeros of
functions with finite Dirichlet integral and some related function
spaces}, Math. Z. \textbf{80} (1962), 217--229.

\bibitem{W} B.~M. Wilson, \emph{Proofs of some formulae enunciated
by  Ramanujan}, Proc. London Math. Soc. \textbf{21} (1922),
235--255.


\end{thebibliography}
\end{document}